\documentclass[11pt,francais]{smfart}
\usepackage[latin1]{inputenc}
\usepackage[T1]{fontenc}
\usepackage[english,francais]{babel}
\usepackage{amssymb,url,xspace,smfthm}
\usepackage[active]{srcltx}
\usepackage[all]{xy}
\usepackage[counterclockwise]{rotating}

\usepackage[backref]{hyperref}
\NumberTheoremsIn{theorem}
\NumberTheoremsAs{theorem}

\makeatletter
    \def\ps@copyright{\ps@empty
    \def\@oddfoot{\hfil\small\copyright 2015, \UFC}}
\makeatother

\newcommand{\UFC}{Université de Franche-Comté}
\newcommand{\BibTeX}{{\scshape Bib}\kern-.08em\TeX}
\newcommand{\T}{\S\kern .15em\relax }
\newcommand{\AMS}{$\mathcal{A}$\kern-.1667em\lower.5ex\hbox
        {$\mathcal{M}$}\kern-.125em$\mathcal{S}$}

\newcommand{\finform}[1]{
\ifthenelse{\equal{\ref{detail:#1}}{\ref{bidon}}}{}{\hyperlink{detail:#1}{\hbox{\tiny{\hspace{1mm}$\blacktriangledown$}}}}}

\title{Sur la conjecture de Collatz}
\date {08-07-2016}
\author{Vincent Fleckinger}
\address{Laboratoire de Mathématiques\\
Université de Franche-Comté\\
16 Route de Gray, F-25000 Besançon}

\email{\href{mailto:vincent.fleckinger@univ-fcomte.fr}{vincent.fleckinger@univ-fcomte.fr}}
\urladdr{\href{http://lmb.univ-fcomte.fr/}{http://lmb.univ-fcomte.fr}}
\author{Ibrahim Abdoulkarim}
\address{Laboratoire de Mathématiques\\
Université de Franche-Comté\\
16 Route de Gray, F-25000 Besançon}
\email{\href{mailto:ibrahim.abdoulkarim@univ-fcomte.fr}{ibrahim.abdoulkarim@univ-fcomte.fr}}
\renewcommand\CurrentInput{src/\jobname.tex}
\hyperbaseurl{}
\begin{document}
\maketitle
\begin{abstract} Cet article propose  une généralisation de  la
 conjecture Collatz, connue aussi sous le nom de problème $3n+1$, ou
 conjecture de Syracuse aux éléments du complété $2$-adique de
 ${\bf Q}$, noté ${\bf Q}_2$. L'introduction
 d'une isométrie $\phi_{2,3}$ de ${\bf Q}_2$
 adaptée au problème permet la reformulation de la conjecture  sous la
 forme $\phi_{2,3}({\bf Q})={\bf Q}$. 
 L'inclusion  $\phi_{2,3}^{-1}({\bf Q})\subset{\bf Q}$, qui est immédiate, donne la forme des rationnels vérifiant la
 conjecture, et leur densité dans ${\bf
 Q}_2$. 
 On étudie ensuite des variations sur le
 couple $(p,q)$ des nombres  
 intervenants dans la fonction que l'on
 itère, et on étudie les éléments du  complémentaire de
 $\bf Q$ dans $\phi_{p,q}({\bf Q})$.  On
 montre ainsi que les couples $(p,q)$
 vérifiant $q^{p-1}<p^p$ sont des candidats
 naturels pour vérifier $\phi_{p,q}({\bf
 Q})={\bf Q}$.
\end{abstract}

\setcounter{section}{-1}\section{Introduction}
Dans la suite une suite $(u_n)_{n\in{\bf N}}$ est dite ultimement périodique s'il
existe un entier $n_0$ tel que la suite
$(u_{n_0+n})_{n\in{\bf N}}$ soit périodique.

On considère l'application $f$ de ${\bf N}$
dans ${\bf N}$ définie par 
\[\forall n\in{\bf N}, f(n)=\left\{\aligned 
&	\frac{n}{2},\text{ si }2|n,\\
&\frac{3n+1}{2},\text{ sinon,}\endaligned
\right.\]
et la suite $u=(u_n)_{n\in{\bf N}}$ définie par $u_0$ et la relation de récurrence
\[\forall n\in{\bf N}, u_{n+1}=f(u_n).\]
Cette suite est une variante de la suite de
Syracuse classique, et permet d'énoncer la
conjecture de Collatz sous la forme suivante :
 \begin{conj} Soit $u$ un entier positif,
	 il existe un entier $n_0$ vérifiant $u_{n_0}=1$. La suite est alors ultimement
  périodique et vérifie pour tout entier naturel $k$, $u_{n_0+2k}=1$ et $u_{n_0+2k+1}=2$.
\end{conj}
L' application $f$ se prolonge naturellement
en une application continue de l'ensemble des
entiers $2$-adiques ${\bf Z}_2$ dans
lui-même. 

On peut généraliser le problème en utilisant des
applications construites sur le même principe.

On considère pour un nombre premier $p$, la
fonction $\epsilon_0$ définie par
$$\epsilon_0:{\bf Z}_p\mapsto
\{0,1,\dots,p-1\}, \epsilon_0(u)\equiv u \mod
p$$ 
et pour  un entier $q$ non nul, la   fonction $g_{p,q}(n)$ définie par
$$\forall u\in {\bf Z}_p,
g_{p,q}(n)\begin{cases}
\frac{n}{p},&\text{si $p|n$},\cr
\frac{qn+\epsilon_0(-qn)}{p},&\text{ sinon}
\end{cases}
.$$

On s'intéresse au  comportement de la suite
$(g_{p,q}^n(u))_{n\in{\bf N}}$, sachant que 
le cas $(p,q)=(2,3)$ redonne le problème de départ.

On obtient alors un premier  résultat  :
\begin{theo} Soit $p$ un nombre premier, $q$ un entier vérifiant
 $1<q<p$ et   $u$ un élément de ${\bf Z}_p$.
 L'élément $u$ est rationnel si et
 seulement si la suite
 $(g_{p,q}^n(u))_{n\in{\bf N}}$ est
 ultimement périodique.
\end{theo}

Introduisons la définition suivante :
\begin{defi}
	Une application $\psi$ de $\bf N$ dans
	$\overline{\bf N}={\bf N}\cup\{+\infty\}$ est dite
	admissible, si elle vérifie les propriétés
	suivantes :
	\begin{enumerate}
		\item 
			L'ensemble ${\bf N}_\psi=\psi^{-1}({\bf N})$ est
			de la forme $I_N=\{i\in {\bf N}, i\leq	N\}$ ou $\bf N$.
		\item  
			L'application $\psi$ est strictement croissante 
			sur l'ensemble ${\bf N}_\psi)$.
	\end{enumerate}
\end{defi}
 
On peut alors reformuler le résultat
précédent sous la forme  suivante.
\begin{coro}\label{theoprin}
 	Soit $p$ un nombre premier, $q$ un entier vérifiant $1<q<p$,
	$\psi$ une fonction admissible de ${\bf N}$
	dans $\overline {\bf N}$, et
	$(a_{n})_{n\in{\bf N}_\psi}$ une famille
	d'entiers à valeur dans $\{1,\dots,p-1\}$.
	La série entière 
	$G_\psi(T)=\sum_{n\in{\bf N}_\psi}^\infty a_{n}p^{\psi(n)}T^n$ 
	est rationnelle si et seulement si $G(1/q)$ est un rationnel de 
	${\bf Q}_p$.
\end{coro}

Une généralisation élémentaire à un corps de nombre $K$, permet
d'obtenir le corollaire suivant :

\begin{coro}\label{theo2} 
	Soit $\psi$ une fonction admissible de ${\bf N}$ dans
	$\overline{\bf N}$.
	La série entière
	$G_\psi(T)=\sum_{n\in{\bf N}_\psi}2^{\psi(n)}T^n$ 
	est rationnelle si et seulement si $G(1/\sqrt{3})$ est dans 
	${\bf Q}(\sqrt{3})\cap {\bf Z}_2[\sqrt{3}]$.
\end{coro}

Ces résultats donne naturellement naissance à
la question suivante :

\begin{enonce*}{Question}
 	Soit $p$ un nombre premier, $q$ un entier positif non nul,
	$\psi$ une fonction admissible de
	${\bf N}$ dans $\overline{\bf N}$, et
	$a_{n}$ une suite à valeur dans $\{1,\dots,p-1\}$.
	La série entière
	$G_\psi(T)=\sum_{n\in{\bf N}_\psi}a_{n}p^{\psi(n)}T^n$ 
	est-elle rationnelle si et seulement si
	$G_\psi(1/q)$ est un rationnel de ${\bf Q}_p$.
\end{enonce*}
Cette question dans le cas $(p,q)=(2,3)$ est alors une reformulation de la conjecture de Collatz.

Un élément de réponse est donné dans la
partie $5$, les couples $(p,q)$ pour lesquels
la réponse pourrait être positive sont ceux
vérifiant $q^{p-1}<p^p$.

\vskip 0.5cm

{\bf \sc  Plan de la suite de  l'article}
\begin{enumerate}
	\item[-]
		La première partie introduit les rappels
		sur les entiers
		$p$-adiques. En particulier le critère de
		rationalité d'un élément $p$-adique en
		fonction de son développement de Hensel.
	\item[-]
		La deuxième partie  donne la construction de
		l'isométrie $\phi_{p,q}$
		de ${\bf Q}_p$ adaptée au problème,
		et énonce le problème général.
	\item[-] 
		 La troisième partie développe la
		 combinatoire reliée au problème.
	\item[-] 
		La quatrième partie est une étude de  la périodicité 
		des éléments de $\phi_{p,q}^{-1}({\bf Q})$. 
	\item[-] 
		La cinquième partie étudie les rayons de
		convergences des séries introduites, et
		les éléments non rationnels de l'image de
		$\phi_{p,q}(Q)$. On obtient aussi des
		résultat sur le comportement en moyenne
		des $m$-premiers termes de  la suite
		$(u_n)$, lorsque $|u|$ est assez grand et
		décrit les classes modulo $p^m$.
	\item  
		La sixième partie traite des exemples
		dans le cas d'un anneau d'entiers
		algébriques monogène, et  de  ${\bf
		F}_p[[T]]$.
\end{enumerate}

\section{L'anneau des entiers $p$-adique ${\bf Z}_p$}

Le lecteur trouvera aisément des références sur les anneaux $p$-adiques,
par exemple le livre d'Yvette Amice sur ce
sujet~\cite{MR0447195}.

La valuation ultramétrique $p$-adique d'un entier $n$, 
définie par $$\forall n\in{\bf Z}, \nu_p(n)=\sup\{k,p^k|n\}$$ à valeur dans ${\bf N}\cup\{+\infty\}$ permet de définir le complété ${\bf Z}_p$ de ${\bf Z}$ pour la distance ultramétrique $|x-y|_p=p^{-\nu_p(x-y)}$. Pour cette métrique une série converge si et seulement si son terme général tend vers $0$. L'anneau ${\bf Z}_p$ est local d'idéal maximal $p{\bf Z}_p$.

 On  prolonge de façon naturelle, par
 multiplicativité, la valuation à ${\bf Q}$, dont le complété $p$-adique est noté ${\bf Q}_p$.

\begin{enonce*}{Notation}
	Afin de rappeler la métrique que l'on considère, 
	lorsqu'on écrit une égalité faisant intervenir 
	un passage à la limite, par exemple la somme 
	d'une série, on notera $=_\infty$ ou $=_p$ au lieu du signe $=$
	réservé aux égalités dans ${\bf Q}$ ou ${\bf Q}[[T]]$
	ne faisant pas intervenir un complété de ${\bf Q}$ particulier.
\end{enonce*}

	Ainsi tout élément de ${\bf Z}_p$ admet un développement 
	sous forme d'une série :

\begin{prop}[Développement de Hensel]
	Tout élément $u$ de ${\bf Z}_p$ admet un développement de 
	Hensel unique :
	$u=_p\sum_{n=0}^\infty a_ip^i$ où la suite $(a_i)$ 
	prend ses valeurs dans $\{0,\dots,p-1\}$.
\end{prop}

On peut faire plusieurs  remarques importantes sur ce développement :

\begin{enumerate}
	\item{}
		La suite $(a_i)_{i\in{\bf N}}$ est ultimement périodique si et
		seulement si $u$ est dans ${\bf Q}\cap {\bf Z}_p$. 
		Si $u=\frac{a}{b}$ est un rationnel écrit sous la forme réduite 
		$(a,b)=1$ et $b>0$, la période correspond alors à l'ordre de $p$
		modulo $b$. Ainsi
			$$-1=_p\sum_{n\in{\bf N}} (p-1)p^n.$$
	\item{}
		L'application $\theta : \{0,1,\dots, p-1\}^{\bf N}\mapsto {\bf Z}_p$
		définie par 
		$\theta(x)=\sum_{i\in{\bf N}}x_ip^i$ 
		est un homéomorphisme, et même une isométrie si l'on prend sur 
		$\{0,1,\dots, p-1\}^{{\bf N}}$ la
		distance définie par  
		\[d(x,y)=p^{-h(x,y)}\text{ où
		}h(x,y)=\inf\{i, x_i\neq y_i\},\]
		avec la convention $h(x,x)=+\infty$.
\end{enumerate}

\section{Construction de l'isométrie} 
\begin{prop}
	La fonction $g_{p,q}$ est uniformément continue pour 
	la topologie $p$-adique.
\end{prop}

En effet  on a 
$$\forall x,y\in{\bf Z}_p,|x-y|_p<1\Rightarrow|g_{p,q}(x)-g_{p,q}(y)|_p=p|x-y|_p$$
et
 $$\forall x,y\in{\bf Z}_p,|x-y|_p=1
\Rightarrow
|g_{p,q}(x)-g_{p,q}(y)|_p\leq|x-y|_p.$$
donc $g_{p,q}$ est lipschitzienne.

Dans la suite on notera $\chi_p$ la fonction caractéristique de  
${\bf Z}_p\backslash p{\bf Z}_p$. Alors  la
quantité $r_k$  définie par 
$$r_0=0,\forall k\in{\bf N}\backslash \{0\}, r_k=\sum_{i=0}^{k} \chi_p(g_{p,q}^i(v)),$$
compte le nombre de termes $u_i$
premiers avec $p$, d'indice inférieur où égal à $k$. 

 \begin{prop}
 Soit $u$ un élément de ${\bf Z}_p$, et
  $u_n=g_{p,q}^n(u)$. La suite $(u_n)_{n\in{\bf N}}$ vérifie la
  relation de récurrence suivante :
 $$\label{recurrence}\forall n\in{\bf N},
 pu_{n+1}=q^{\chi_p(u_n)}u_n+\epsilon_0(-qu_n).$$
En particulier on obtient les relations suivantes dans ${\bf Z}_p$ :
\begin{equation}
\label{formule1}\forall n\in{\bf N},
u+\sum_{i=0}^n
\frac{\epsilon_0(-qu_i)p^i}{q^{r_i}}=_p\frac{u_{n+1}p^{n+1}}{q^{r_n}}.
\end{equation}
et par passage à la limite
\begin{equation}\label{formule2}\forall
	n\in{\bf N}, u+\sum_{i\in{\bf N}}
	\frac{\epsilon_0(-qu_i)p^i}{q^{r_i}}=_p0.
\end{equation}
\end{prop}

\begin{prop}
	L'application $\phi_{p,q}:{\bf Z}_p\mapsto {\bf Z}_p$ 
	définie par 
		$$\forall u\in{\bf Z}_p, 
		\phi_{p,q}(u)=_p\sum_{n=0}^\infty \epsilon_0(-qu_n)p^n$$ 
	est  une isométrie de ${\bf Z}_p$.
\end{prop}

En effet un calcul immédiat donne, si
$u=v+p^nw$ et $|w|_p=1$, pour
tout entier $k$ inférieur ou égal à $n$,
$g_{p,q}^k(u)=g_{p,q}^k(v)+p^{n-k}q^{r_{k-1}}w$.

En particulier le premier indice pour lequel $u_k$ et $v_k$ 
ne sont pas dans la même classe modulo $p$  est $n$, d'où
$|\phi_{p,q}(u)-\phi_{p,q}(v)|_p=|u-v|_p$.

 D'après la proposition précédente, et la formule~(\ref{formule2}),
 $\phi_{p,q}$ est  surjective, d'application réciproque définie par :

 $$\forall v\in {\bf Z}_p, 
 		v=_p\sum_{i=0}^\infty a_i p^i,\qquad 
		\phi_{p,q}^{-1}(v)=_p
		-\sum_{i=0}^\infty\frac{a_ip^i}{q^{r_i}}.$$
où $v=_p\sum_{i} a_i p^i$ est le développement de Hensel de $v$.

\bigskip

Notons $\delta_p$ l'application $\frac{1}{p}(Id-\epsilon_0)$. C'est une application continue qui    permet de tester la périodicité d'un développement de Hensel :

\begin{prop} 
	Pour tout développement de Hensel définie par une suite $a:{\bf N}\mapsto \{0,1\}$, on a l'égalité suivante :
$$ \delta_p(\sum_{i=0}^\infty a_ip^i)=_p\sum_{i=0}^\infty a_{i+1}p^i.$$
\end{prop}

\noindent
L'application $\delta_p$ permet donc de  tester la rationalité  d'un
élément de ${\bf Z}_p$, puisque d'après les
propriétés du développement de Hensel, 
$$\forall v\in {\bf Z}_p, v\in {\bf Q}\Leftrightarrow \exists k\in{\bf N}^*,\exists n_0,\forall n\geq n_0, \delta_p^{n+k}(v)=_p\delta_p^n(v).$$
Un entier $k$ vérifiant la propriété précédente est appelé une période
du développement.

\bigskip

La définition de $\phi_{p,q}$ permet d'obtenir sans difficulté la relation suivante :
\begin{prop}
On a $\phi_{p,q}\circ g_{p,q}=\delta_p\circ \phi_{p,q}$. 
\end{prop}
Cette  égalité s'itère de façon naturelle en :
$$\forall n\in{\bf N}, \phi_{p,q}\circ g_{p,q}^n=\delta_p^n\circ \phi_{p,q}.$$

On en déduit immédiatement le corollaire :

\begin{coro}
$$\phi_{p,q}^{-1}({\bf Q}\cap{\bf Z}_p)=\{u\in{\bf Z}_p, \exists n_0,\forall n\geq n_0,\exists k\in{\bf N}, g_{p,q}^{n+k}(u)=g_{p,q}^n(u)\}.$$
\end{coro}
Ainsi on a les valeurs particulières
\begin{align*}
\phi_{p,q}^{-1}(-1)=_
p-\sum_{n=0}^\infty\frac{(p-1)p^i}{q^{i+1}},\quad
\phi_{2,3}^{-1}(-\frac{1}{3})=_2-\sum_{n=0}^\infty \frac{2^{2i}}{3^{i+1}}
\end{align*}
soit encore 
\begin{align*}\phi_{p,q}^{-1}(-1)=_p\frac{1-p}{q-p}
,\quad\phi_{2,3}^{-1}(-\frac{1}{3})=_21.
\end{align*}

\begin{prop} On a l'inclusion suivante :
$$\phi_{p,q}^{-1}({\bf Q}\cap{\bf Z}_p)\subset {\bf Q}\cap{\bf Z}_p.$$
\end{prop}
La démonstration de cette proposition, qui repose sur la combinatoire des séries formelles associées au problème, sera faite dans la troisième partie.

\begin{enonce*}{Question}[Problème de Collatz généralisée]
Soit $\phi_{p,q}$ l'isométrie de ${\bf Q}_p$ obtenue en prolongeant celle de ${\bf Z}_p$ par 
$$\forall u\in{\bf Z}_p,\forall n\in {\bf N},\phi(p^{-n}u)=p^{-n}\phi(u).$$
On peut alors poser  la  question  générale :
Pour quels couples $(p,q)$ a-t-on l'égalité $\phi_{p,q}({\bf Q})={\bf Q}$?
\end{enonce*}

\smallskip\noindent
Le résultat précédent donne  l'inclusion :
$$\phi_{p,q}^{-1}({\bf Q})\subset{\bf Q}.$$
La continuité de $\phi_{p,q}$ permet d'affirmer que l'ensemble des
rationnels pour lesquels la suite est
ultimement périodique est dense dans ${\bf
Q}_p$.

En particulier pour le couple $(2,3)$,
l'ensemble des
rationnels vérifiant la conjecture de Collatz
sont denses dans
${\bf Q}_2$.

Si $q<p$  la réponse est positive du fait que
la hauteur des éléments de la suite des
itérés reste bornée. 
Pour le cas $p<q$ des simulations numériques
montre que la réponse semble être
assez souvent négative.

\section{Combinatoire de la suite de Collatz}
Dans la suite on étudie le cas $(p,q)$ en utilisant les notations suivantes
:

\[\phi_{p,q}(u)=_p\sum_{n=0}^\infty a_np^n,\]
\[\forall n\in{\bf N}, g_{p,q}^{n}(u)=u_n,\,
r_n=\mathrm{Card}(\{0\leq i\leq n-1,
a_i=0\}\]
Et on introduit la fonction
$\psi_u(n)$ de ${\bf N}$ dans ${\bf N}\cup\{+\infty\}$ définie par :
\[\phi_{p,q}(u)=_p\sum_{n\in {\bf
N}}a_{\psi_u(n)}p^{\psi_u(n)}.\]
En particulier $\psi_u(n)$ est l'indice du
$n$-ième coefficient non nul dans le
développement de Hensel de
$\phi_{p,q}(u)$.
La fonction $\psi_u$ est une
fonction admissible de $\bf N$ dans
$\overline{\bf N}$.

On remarquera, que $+\infty$ est dans l'image de $\psi_u$ si et
seulement si la fonction la suite $r_n(u)$
est bornée, donc constante à partir d'un
certain rang. De plus, avec les notations introduites on obtient  :
\[\forall n\in{\bf N},\psi_u(n)\in{\bf
N}\Rightarrow  r_{\psi_u(n)}=n+1,\]
 \[p^{+\infty}=_p0\hbox{ et
 }\sum_{n=0}^\infty
a_np^n=_p\sum_{n=0}^\infty
 a_{\psi_u(n)}p^{\psi_u(n)}.\]
Pour $u$ fixé on notera $\psi=\psi_u$.

\begin{prop} Soit $u$ un élément de ${\bf Z}_2$, alors 
les séries entières
$$F_\psi(T)=\sum_{n=0}^\infty
a_{\psi(n)}p^{\psi(n)}T^n\hbox{ et
}G_u(T)=\sum_{n=0}^\infty
a_n\frac{p^{n}}{q^{r_n}}T^n$$
vérifie les égalités suivantes
:
\[\frac{u}{1-T}+\frac{1}{1-T}G_u(T)=\sum_{n=0}^\infty\frac{p^{n+1}u_{n+1}}{q^{r_n}}T^n.$$ 
$$\frac{qu}{1-qT}+\frac{1}{1-qT}F_\psi(T)=\sum_{n=0}^\infty
p^{\psi(n+1)}u_{\psi(n+1)}T^n,\]
\end{prop} 
La première provient de la proposition 10. 

La deuxième provient des égalités :
\begin{equation}
\label{exp}
\forall n\in{\bf N},
u+\sum_{i=0}^na_{\psi(i)}\frac{p^{\psi(i)}}{q^{i+1}}=_p\frac{p^{\psi(n+1)}u_{\psi(n+1)}}{q^{n+1}}.
\end{equation}
 et  d'obtenir l'égalité des séries entières. 

%

\begin{prop} Soit $u$ un élément de ${\bf
	Z}_p$, et $\psi_u$ la fonction associée.
La série entière
$$S_u(T)=_2\sum_{n=0}^\infty p
^{{\psi}(n+1)}u_{\psi(n+1)}T^n.$$ est une
fraction rationnelle
si et seulement si $\phi_{p,q}(u)$ est un rationnel.
\end{prop}
En effet la deuxième égalité de la
proposition précédente, permet d'affirmer que
si $S_u$ est une fraction rationnelle, alors
$F_\psi$ en est aussi une, et ses
coefficients sont dans $\bf Q$, puis
$\phi_{p,q}(u)=_pF_\psi(1)$
est un rationnel, donc son développement de
Hensel est périodique, et enfin
$u=-\frac{1}{q}F_\psi(\frac{1}{q})$ est aussi
un rationnel, ainsi que les $u_{\psi(n)}$.
Réciproquement si $\phi_{p,q}(u)$ est un
rationnel, son développement de Hensel
$F_\psi(1)$ est ultimement périodique, donc
la série $F_\psi$ est une fraction
rationnelle, ainsi que $S_u$.
\begin{coro}
	$$\phi_{p,q}^{-1}({\bf Q})\subset{\bf Q}.$$
\end{coro}
En effet il suffit de se ramener au cas où
$\phi_{p,q}(u)$ est un entier $p$-adique
rationnel, comme $\phi_{p,q}(u)=F_{\psi_u}(1)$, $F_{\psi_u}$ est une fraction rationnelle, puis
 sachant que $F_{\psi_u}$ converge sur un
 disque de rayon au moins $p$ dans ${\bf
 Q}_p$, on peut écrire
 $u=_p -\frac{1}{q}F_{\psi_u}(\frac{1}{q})$, ce qui prouve sa
rationalité.

\bigskip

On peut reformuler la conjecture de Collatz sous la forme :
\begin{conj}
Soit $\psi$ une fonction admissible de ${\bf
N}$ dans $\overline{\bf N}$. La série entière 
$G_\psi(T)=\sum_{n=0}^\infty 2^{\psi(n)}T^n$
est rationnelle si et seulement si
$G_\psi(\frac{1}{3})$ est un rationnel de ${\bf Q}_2$.
\end{conj}

Le problème général correspondant à un couple
$(p,q)$ introduit une combinatoire similaire,
et admet une solution élémentaire
dans le  cadre suivant :
\begin{theo} Soit $p$ un nombre premier, $q$ un entier vérifiant $1<q<p$ et
$\psi$ une fonction admissible 
de ${\bf N}$ dans $\overline{\bf N}$, et $a_{n}$ une
suite à valeur dans $\{1,\dots,p-1\}$.
La série entière 
$G_\psi(T)=\sum_{n=0}^\infty
a_{n}p^{\psi(n)}T^n$ est rationnelle si et
seulement si $G_\psi(1/q)$ est un rationnel de ${\bf Q}_p$.
\end{theo}
\begin{defi}
	Soit $\frac{a}{b}$ un rationnel sous forme
	réduite, on note appelle hauteur de $u$ 
	la quantité $h(u)=\max\{|a|,|b|\}$.
\end{defi}
Seule la réciproque pose problème.
On remarquera ici que l'indexation de la
suite $a_n$ a changée, et que la suite  est à
valeur dans $\{1,\dots,p-1\}$, donc peut
s'écrire sous la forme $a'_{\psi(n)}$ pour
une suite $(a'_n)$ dont le $n$-ième terme non
nul est $a_n=a'_{\psi(n)}$.
Lorsque $0<q<p$, la suite de rationnels obtenue à
partir d'un rationnel $u$,  par
itération de $g_{p,q}$ admet une hauteur
bornée. Elle   prend un nombre fini de valeurs,
donc est ultimement périodique. D'où la
rationalité de la série.

\section{Points périodiques}
\begin{defi}
Un élément $u$ de ${\bf Q}_p$ est dit $k$-périodique
pour $g_{p,q}$ si la suite
$(g_{p,q}^n(u))_{n\in{\bf N}}$ est
périodique, de période $k$, c'est-à-dire si
$g_{p,q}^k(u)=u$.
\end{defi}
\begin{prop}
Un élément $u$ de ${\bf Q}_p$ est
$k$-périodique si et seulement si il existe
un entier positif  $k$ tel que 
$(1-p^k)g_{p,q}(u)$ est un entier
vérifiant $0\leq
(1-p^k)g_{p,q}(u)<p^k$. 
\end{prop}
\begin{proof}
C'est une conséquence immédiate de la
description des éléments de ${\bf Z}_p$ ayant
un développement de Hensel périodique.
\end{proof}
Le plus petit entier positif $k$ tel que $u$ soit
$k$-périodique,  est appelé la période de $u$.
\begin{coro}
Il y a exactement $p^k$ éléments $u$ de ${\bf
Q}_p$ vérifiant l'égalité $g_{p,q}^k(u)=u$. 
Plus précisément si $\ell$ est un entier
positif inférieur à $k$,
$(a_i)_{0\leq i\leq \ell-1}$ un élément de  
$\{1,\dots,p-1\}^\ell$ et
$\psi$ une application strictement croissante de
$\{0,\dots,\ell\}$ dans $\{0,\dots,k\}$
vérifiant  et $\psi(\ell)=k$, alors
$$u=-\frac{\sum_{i=0}^{l-1}
a_ip^{\psi(i)}q^{\ell-i-1}}{q^\ell-p^k}.$$
\end{coro}

En effet on a alors
\[u=_p-\sum_{i=0}^{\ell-1}a_i\frac{p^{\psi(i)}}{q^{i+1}}
\frac{1}{1-\frac{p^k}{q^{\ell}}}\]
D'où
\[\phi_{p,q}(u)=_p\frac{1}{1-p^k}\sum_{i=0}^{\ell-1}a_ip^{\psi(i)}\]
On obtient ainsi tous les développements de
Hensel périodiques de période $k$.
\begin{defi}
Un élément $u$ de ${\bf Q}_p$ est dit
ultimement $k$-périodique
pour $g_{p,q}$ si la suite
$(g_{p,q}^n(u))_{n\in{\bf N}}$ est 
périodique, de période $k$ à partir d'un
certain rang, c'est-à-dire s'il existe un
entier $n_0$ tel que
$g_{p,q}^{n_0+k}(u)=g_{p,q}^{n_0}(u)$.
\end{defi}
\begin{prop}
Un élément $u$ de ${\bf Q}_p$ est
ultimement $k$-périodique si et seulement si 
$(1-p^k)g_{p,q}(u)$ est un entier
positif.
\end{prop}
On peut s'intéresser aux points périodiques
entiers. Une façon simple de fabriquer des
période entière est d'imposer que le
dénominateur qui intervient dans la formule
soit égal à $1$. Pour cela la conjecture de Catalan
démontrée par P. Mihail\u escu\cite{MR2185753} donne une
précision
sur les dénominateurs qui interviennent dans
l'expression des éléments $k$-périodiques.
\begin{theo}(Conjecture de Catalan)
Soient $n$ et $m$ deux entiers supérieurs ou
égaux à $2$.
L'équation $x^m-y^n=1$ n'admet qu'un nombre
fini de solutions,
$$(m,n,x,y)\in\{(2,3,-3,2),(2,3,3,2)\}.$$
\end{theo}
On en déduit que pour $p>3$ les seules
possibilités conduisant à $q^\ell-p^k=\pm1$
avec $\ell<k$ sont obtenues pour $\ell=1$,
$q=p^k\pm1$.  
 
On peut donc  résoudre complètement
les équations $3^\ell-2^k=-1$ et
$3^\ell-2^k=1$.
\begin{coro}
Considérons l'équation
$3^\ell-2^k=\pm1$.
\begin{enumerate}
\item{}L'ensemble des couples $(k,\ell)$  solutions de
$3^\ell-2^k=-1$
est $\{(1,0),(2,1)\}$.
\item{}L'ensemble des couples $(k,\ell)$ solution de
$3^\ell-2^k=1$ est $\{(1,1),(3,2)\}$.
\end{enumerate}
\end{coro}

Le cas $(1,0)$
conduit au point fixe $0$, le cas $(2,1)$
donne les $2$-périodes  $1$ et
$2$. 

Le couple $(1,1)$  donne le point fixe $-1$,
et le couple  $(3,2)$ donne  les points
$3$-périodiques entiers négatifs
$(-5, -7,-10)$.
La conjecture de Collatz prédit que les
seules périodes entières positives sont $0$
et $1$.
Par contre il y a plusieurs périodes
entières négatives, ne correspondant pas aux solutions
de l'équation de Catalan, on obtient ainsi le cycle 
$(-17,-25,-37,-55,-82,-41,-61,-91,-136,-68,-34)$.
pour ces valeurs on a $k=11$, $l=7$,
$3^7-2^{11}=139$ et la relation de divisibilité
:
\[17=\frac{3^6+2\cdot 3^5+2^2\cdot
3^4+2^3\cdot
3^3+2^5\cdot3^2+2^6\cdot3+2^7}{3^7-2^{11}}\]
On peut écrire une relation analogue pour
chaque période du cycle précédent.

On peut donc se poser naturellement 
la question de l'existence d'un entier
strictement supérieur à $2$ qui soit une
$k$-période. La conjecture de Collatz dans le
cas $(2,3)$ classique entraine une réponse négative.
\medskip
On s'intéresse maintenant à la l'adhérence
des entiers vérifiant la conjecture de
Collatz dans le cas $(2,3)$.

\begin{lemm}
Le groupe ${\bf Z}_3^*$ des éléments
inversibles de ${\bf Z}_3$ est engendré
topologiquement par $2$.
\end{lemm}

\begin{proof}
Notons $U_3={\bf Z}_3^*$ le groupe des
éléments inversibles de ${\bf Z}_3$,
$U_3^{(k)}=1+3^k{\bf Z}_3$
les noyaux des surjections naturelles sur
$({\bf Z}/3^k{\bf Z}_3)^*$.
Soit $\pi$ désigne la surjection naturelle de
$U_3$ dans $({\bf Z}_3/{\bf Z}_3)^*$ et, pour tout entier $k$
supérieur ou égal à $1$, $\pi_k$ désigne
l'homomorphisme de groupe 
\begin{eqnarray*}
	\pi_k&:& U_3^{(k)}\to {\bf Z}_3/3{\bf Z_3}\cr
\pi_k(u)&=&\frac{u-1}{3^k}
\end{eqnarray*}

Ainsi on a ainsi  les suites
exactes 
$$\{1\}\to U_3^{(1)}\to U_3\to ({\bf Z}_3/{\bf
Z}_3)^*\to\{1\}.$$ 
$$\{1\}\to U_3^{(k+1)}\to U_3^{(k)}\buildrel
\pi_k\over\to {\bf Z}_3/{\bf
Z}_3\to\{0\}.$$ 

De plus le caractère de Teichmuller en $2$
est donné par
\begin{eqnarray*}
\chi_3(2)&=_3&\lim_{n\to\infty}2^{3^n}\cr
&=_3&-1
\end{eqnarray*}
Donc $U_3$ se décompose en produit
direct $U_3=\{\pm1\}\times U_3^{(1)}$.

Comme $4=1+3$, $4$ engendre le
groupe multiplicatif
$U_3^{(1)}/U_3^{2}$ qui est un ${\bf F}_3$ espace
vectoriel de dimension $1$. D'après  le
 lemme de Nakayama, l'élément $2^2$
 engendre $U_3^{(1)}$. 
Donc l'adhérence de $2^{\bf Z}$ contient
$\{\pm1\}\times U_3^{(1)}$, c'est-à-dire
${\bf Z}_3^*$.
\end{proof}
\begin{coro}
Les entiers qui vérifient la conjecture de
Collatz sont dense dans ${\bf Z}_2$.
\end{coro}
\begin{proof}
Soit $u$ un entier, posons
$\phi_{2,3}(u)=\sum_{i\in{\bf N}}
2^{\psi(i)}$.
Alors pour tout entier $n$,
$3^nu+\sum_{i=0}^{n-1}
2^{\psi(i)}3^{n-i}=2^{\psi(n)}u_{\psi(n)}.$

Soit $$\psi'(n)=\min\{{k\in{\bf N}},\,k>\psi(n-1),\,
|2^k-\sum_{i=0}^{n-1}2^{\psi(i)}3^{n-i}|_3
<\frac{1}{2^{\psi(n)}}\},$$ 
et
$$v_n=\sum_{i=0}^{n-1}2^{\psi(i)}+2^{\psi(n)}/(1-4),$$
alors $-3v_n$ est dans ${\bf N}$ et
$$\phi_{2,3}^{-1}(v_n)=\frac{2^{\psi(n)}-\sum_{i=0}^{n-1}2^{\psi(i)}3^{n-i}}{3^n}$$
est un entier.
Comme
$|u-\phi_{2,3}^{-1}(v_n)|_2=|\phi_{2,3}(u)-v_n|_2$

on obtient
$$|u-\phi_{2,3}^{-1}(v_n)|_2\leq
\frac{1}{2^{\psi(n)}}.$$
D'où le résultat.
\end{proof}

On peut préciser ce résultat en utilisant le
théorème des sous-espaces établi par Schmidt
\cite{MR568710} et sa version $p$-adique
donnée par Schlickewei \cite{MR0429784}
\begin{theo}
	On note $\overline{\bf Q}$ la clôture
	algébrique de $\bf Q$ dans $\bf C$.
On considère un entier $m$ supérieur ou égal
à 2, et un réel strictement positif
$\epsilon$. Soient $L_1$,\dots,$L_m$ une base
du dual de $\overline{\bf Q}^m$.
 L'ensemble des solutions entières 
${\bf x}
=(x_1,\dots,x_m)\in{\bf Z}^m$  de l'inégalité
$$\prod_{i=1}^m|L_i({\bf x})|\leq
max\{|x_1|,\dots,|x_m|\})^{-\epsilon}$$
est contenu dans une union finie de
sous-espaces vectoriels propres de ${\bf
Q}^m$.
\end{theo} 
\begin{theo}
On considère un entier $m$ supérieur ou égal
à 2, un réel strictement positif
$\epsilon$, et un ensemble fini de  nombres
premiers $S$. Pour chaque place $\nu$
appartenant à $S\cup\{\infty\}$, on considère
des formes linéaires $(L_{i,\nu})_{1\leq
i\leq \nu}$ sur $\overline{\bf
Q}^m$, linéairement indépendantes sur ${\bf
Q}$. L'ensemble des solutions entières
${\bf x}
=(x_1,\dots,x_m)\in{\bf Z}^m$  de l'inégalité
$$\prod_{\nu\in
S\cup\{\infty\}}
\prod_{i=1}^m|L_{\nu,i}({\bf x})|_\nu\leq
max\{|x_1|,\dots,|x_m|\})^{-\epsilon}$$
est contenu dans une union finie de
sous-espaces vectoriels propres de ${\bf
Q}^m$.
\end{theo}
 
On en déduit le théorème suivant :

\begin{theo}
	Soit $u$ un rationnel, on pose 
	$$\phi_{2,3}(u)=\sum_{n\in{\bf
	N}}2^{\psi(n)},$$
	et $u_n=\phi_{2,3}^n(u)$. Si $\omega$
	est un rationnel, 
	on note $$\psi'_\omega(n+1)=\inf
	\{k\in{\bf
	N}|\sum_{i=0}^n2^{\psi(i)}3^{n-i}-\omega 2^{k}|_3\leq
	\frac{1}{3^{n+1}}\},$$ alors  la suite
	$(u_n)_{n\in{\bf N}}$ est ultimement
	périodique si et seulement si il existe
	un rationnel $\omega$ tel que
	$\liminf \frac{\psi'_\omega(n)}{n}<\infty$.
	Le rationnel $\omega$ est alors un point
	périodique.
\end{theo}

\begin{proof}
	Si la suite $(u_n)_{n\in{N}}$ est
	ultimement périodique, il suffit de
	choisir $\omega$ parmi les points
	périodiques de la suite.  

	On pose 
$$\begin{array}{|l|c|c|c|}
\hline
&\infty&2&3\cr
\hline
L_{\nu,1}&x&x&x\cr
\hline
L_{\nu,2}&y&ux+y&y-\omega z\cr
\hline
L_{\nu,3}&z&z&z\cr
\hline
\end{array}
$$
Alors si ${\bf x}_n=(3^{n+1},
\sum_{i=0}^n2^{\psi(i)}3^{n-i},2^{\psi'(n+1)})$,
et $$H(x,y,z)=\max\{|x|,|y|,|z|\},$$
 on a 
$$\prod_{\nu\in S\cup\{\infty\}}
\prod_{i=1}^3
|L_{\nu,i}({\bf x}_n)|_\nu=
\frac{\sum_{i=0}^n
2^{\psi(i)}3^{n-i}}{2^{\psi(n+1)}3^{n+1}}.$$
En particulier 
$$\liminf \left[\prod_{\nu\in S\cup\{\infty\}}
\prod_{i=1}^3
|L_{\nu,i}({\bf
x}_n|_\nu\right]^{\frac{1}{n}}
\leq \frac{1}{2}.$$
Si $\liminf \frac{\psi'_\omega(n)}{n}$ est un réel
$\beta>0$, alors $2^\beta>1$, et pour
$\epsilon<\frac{\log2}{\log 3}$ il existe une
fonction strictement croissante $h$ de  ${\bf
N}$ dans lui-même telle que la sous-suite
infinie ${\bf x}_{h(n)}$ vérifie
$$\prod_{\nu\in S\cup\{\infty\}}
\prod_{i=1}^3
|L_{\nu,i}({\bf x}_{h(n)})|_\nu<H({\bf
x}_{h(n)})^{-\epsilon}$$
Il existe donc des entiers algébriques
$a,b,c$ non tous nuls, tels que, quitte à
raffiner la sous-suite définie par  $h$, 
la forme linéaire sur ${\overline{\bf Q}}^3$
définie par  $L(x,y,z)=ax+by+cz$ vérifie 
$$\forall n\in{\bf
N},\, L({\bf x}_{h(n)})=0.$$ 
Soit encore
$$\forall n\in{\bf
N},\,
a
3^{h(n)+1}+b\sum_{i=0}^{h(n)}2^{\psi(i)}3^{h(n)-i}+c2^{\psi_\omega(h(n)+1)}=0.$$ 
L'élément $b$ est donc non nul et 
$$\forall n\in{\bf
N},\,
\frac{a}{b}
+\sum_{i=0}^{h(n)}\frac{2^{\psi(i)}}{3^{i+1}}=-\frac{c}{b}\frac{2^{\psi'(h(n)+1)}}{3^{h(n)+1}}.$$ 
Par passage à la limite on obtient que
$\frac{a}{b}$ est égal à $u$. De plus la suite  
$g_{2,3}^n(u)$ est ultimement
périodique, puisque l'élément 
$-\frac{c}{b}$ s'identifie alors  à terme de
la suite $(u_n)_{n\in{\bf N}}$ apparaissant
une infinité de fois. On obtient par
identification dans ${\bf Q}_3$
$\omega=_3-\frac{c}{b}$, d'où
$\omega=-\frac{c}{b}$.

\end{proof}
\begin{rema}
	Ce résultat se généralise dès que
	l'on peut  définir la fonction
	$\psi$, ce qui se produit si $p$
	engendre ${\bf Z}_q^*$.
\end{rema}

\section{Rayons de méromorphie des séries}
L'objectif de cette partie est d'utiliser les
informations sur les séries introduites
précédemment, pour l'étude des éléments non
rationnels de
$\phi_{p,q}(\bf Q)$, c'est à dire les
éléments rationnels tels que la suite
$(g_{p,q}^n(u))$ ne soit pas ultimement
périodique. On suppose donc $q$ premier avec
$p$ et supérieur à $p$.
\begin{prop}
Un rationnel $u$ ne vérifie pas la conjecture
de Collatz généralisée si et seulement si $\lim_{n\to\infty}|u_n|_\infty=_\infty+\infty$
\end{prop}
On remarque  que  si $du$ est
entier, alors pour tout entier $n$, $du_n$
est aussi un entier.
Dire que la suite $g_{p,q}^n(u)$ n'est pas
ultimement périodique est donc équivalent au
fait que pour tout entier 
 $N$, il existe un rang $n_0$
au-delà duquel la suite prend des valeurs
$|u_n|>N$, c'est à dire qu'elle admet pour
limite $+\infty$ ou $-\infty$ quand $n$ tend
vers l'infini.
On a donc une équivalence entre $\lim_n |u_n|_\infty=_\infty+\infty$ et le fait que $u$ ne vérifie pas la conjecture de Collatz.

\begin{prop} \label{rayon}Soit $u$ un
	rationnel tel que $\phi_{p,q}(u)$ ne soit
	pas rationnel. 
Le rayon de convergence de la série
$$g(T)=\sum_n \frac{p^{n+1}u_{n+1}}{q^{r_n}} T^{n}$$ pour la métrique 
archimédienne est  égal à $1$. De plus
$$\lim_{n\to\infty} p{u_{n+1}^{{1}/{n}}}{q^{-{r_n}/{n}}}=_\infty 1.$$
\end{prop}

Posons $v_n=p^{n+1}u_{n+1}q^{-r_n}$,
il suffit d'étudier le quotient ${v_{n+1}}/{v_n}$.

On a
\[\frac{v_{n+1}}{v_n}\in\{1,\frac{i}{qu_n},0\leq
i\leq p-1\}\]

La limite de $v_{n+1}/v_n$ vaut donc $1$ dès que  la suite $u_n$ tend vers l'infini pour la métrique archimédienne.
De plus la suite ${v_{n+1}}/{v_n}$ admettant une limite,
 la suite  $pu_{n+1}^{{1}/{n}}q^{-{r_n}/{n}}$ est convergente de même limite, ce qui donne le rayon de convergence égal à 1 pour la métrique archimédienne.

\begin{coro} 
	Si $u$ est un rationnel tel que
	$\phi_{p,q}(u)$ n'est pas rationnel, la
	 série $S_u$ admet un
	rayon de convergence égal à $\frac{1}{q}$,
	et la série
	$F_\psi$,  admet un rayon de convergence
	supérieur ou égal à $1/q$ pour la métrique
	archimédienne.
\end{coro}
 
Le rayon de convergence de la série $S_u$ est
analogue au précédent, et l'inégalité pour
celui de $F_\psi$ provient de l'identité
\[\frac{qu}{1-qT}+\frac{1}{1-qT}F_\psi(T)=S_u(T).\]

Le developpement de Hensel de
$\phi_{p,q}(u)$ codant le comportement de
la suite $g_{p,q}^n(u)$, permet un certain
contrôle sur la croissance de cette suite.
\begin{defi}
	Pour tout rationnel $u$ non nul, notons $H(u)$ le plus grand entier relatif
	vérifiant $p^{H(u)-1}\leq |u|<p^{H(u)}$.
\end{defi}
Dire que  $u$ n'appartient pas à 
$\phi_{p,q}^{-1}({\bf Q})$ se traduit par
le fait que
$H(g_{p,q}^n(u)$ tend vers $+\infty$.

On introduit alors le plus grand entier $r$
vérifiant $q^{r-1}<p^r$. 
Notons  $f_0(u)=\frac{u}{p}$ et pour tout $i$
dans $\{1,\dots , p-1\}$,
$f_i(u)=\frac{qu+i}{p}$. 
Alors si $\alpha_0,\dots,\alpha_{r-1}$ sont
les coefficients correspondant a une tranche
$T$, on pose
\[f_T=f_{\alpha_{r-1}}\circ\cdots\circ
f_{\alpha_0}.\]
Alors si la première tranche du développement
de Hensel de $\phi_{p,q}(u)$ correspond à
$T$, $f^r(u)=f_T(u)$. 
L'application $f_T$ est de la forme
$f_T(u)=\frac{q^{e_T}+\beta_T}{p^r}$ où $e_T$
désigne le nombre de coefficients non nuls
dans $T$, $\beta_T$ est un entier dépendant
que de $T$.

Si $n$ est un entier soit $m$ l'entier défini
par les inégalités
\[rm\leq \psi(n)\le r(m+1)\]
Regardons alors les $m$ premières tranches
$T_k$, $1\leq k\leq m$,
de longueur $r$ du développement de Hensel.
On note $\ell_i$ le nombre de tranches $T_k$
dont exactement $i$ coefficients sont non
nuls.
Alors $m=\sum_{i=0}^r\ell_i$, et si
$n=\sum_{i=0}^r i\ell_i$ alors
$mr\leq \varphi(n)\le r(m+1)$.
\begin{prop}
	Il existe un entier $M$ tel que pour tout
	$u$ un rationnel vérifiant $H(u)>M$, si $T$
	désigne la première
	tranche du développement de Hensel de $\varphi_{p,q}(u)$  alors
	
	\[\left\{\aligned 
		-r+H(u)&\leq H(f^r(u))\leq H(u)-r,\text{si $e_T=0$}\\
		e_T-r+H(u)&\leq H(f^r(u))\leq H(u)+e_T+1-r,\text{ si $0<e_T<r$}\\
H(u)+1&\leq H(f^r(u)\leq H(u)+2, \text{ si $e_T=r$}.
\endaligned\right.\]
	 \end{prop}

On utilise le fait que $f^r(u)=f_T(u)$ et les
majorations
\[p^{H(f_T(u))} > \frac{q^{e_T}p^{H(u)-1}}{p^r}\]
et 
\[q^{e_T}p^{H(u)}+\beta_T\geq
p^{H(f_T(u))-1+r}.\]
Il existe un entier $M$ tel que 
que pour $H(u)>M$
\[q^{e_T}+\beta_Tp^{-H(u))}\geq
p^{H(f_T(u))-H(u)-1+r}\Leftrightarrow
q^{e_T}\geq
p^{H(f_T(u))-H(u)-1+r}.\]
On utilise alors la définition de $r$ pour
conclure.
\begin{coro}
	Soit $u$ un rationnel tel que pour tout
	entier $n$, $H(u_n)>M$, alors 
	\[n-rm+\ell_r\leq H(f^{rm}(u)-H(u)\leq
n+(1-r)m+\ell_r-\ell_0.\]
\end{coro}
Cela provient de la contribution de chaque
tranche $T_k$.

\begin{coro}
Si l'entier $r$ correspondant au couple
$(p,q)$ est égal à $2$, alors
\[H(f^{rm}(u)\leq H(u)+2(\ell_2-\ell_0).\]
\end{coro}

C'est le cas pour le couple $(2,3)$. 

\begin{theo}
	On suppose que le couple $(p,q)$ vérifie
	$p^2>q$ et $p^3<q^2$.
	Soit $u$ un rationnel tel que
	$\phi_{p,q}(u)$ admet un développement de
	Hensel. Si la quantité $\ell_2-\ell_0$
	correspondant à un découpage en tranche de
	longueur $2$ reste bornée supérieurement ,alors $\phi_{p,q}(u)$ est un rationnel.
\end{theo}

En effet la hauteur de la suite
$g_{p,q}^n(u)$ est alors bornée.
\begin{prop}
	Si $u$ est un rationnel dont l'image par
	$\varphi_{p,q}$ n'est pas rationnelle alors 
	$\varlimsup\frac{\ell_r}{n}\leq \frac{\log
	q-\log p}{\log p}$ et
	$\varlimsup\frac{\psi(n)}{n}\leq \frac{\log
	q}{\log p}$. De plus on a l'implication
	\[\varlimsup\frac{\ell_r}{n}=0\Rightarrow
		\varliminf\frac{\psi(n)}{n}\geq
	r\frac{\log q-\log p}{\log p}.\]
\end{prop}

En effet d'après l'étude du rayon de
convergence  de la série $S_u$, on sait que
$\frac{H(f^{rm}(u))+rm}{n}$ admet pour limite
$\frac{\log q}{\log p}$ lorsque $n$ tend vers
l'infini. On en déduit la majoration de
$\varlimsup\frac{\psi(n)}{n}$. La minoration de $H(f^{rm}(u))$
donne la majoration de
$\varlimsup\frac{\ell_r}{n}$, la majoration
de $H(f^{rm}(u))$ donne la minoration de
$\varliminf\frac{\varphi{n}}{n}$ sachant que
$m\leq \frac{1}{r}\varphi(n)\le m+1$.

On remarque que dans ce cas on obtient un
encadrement du rayon de convergence
$p$-adique, $\rho_p$
des séries $S_u$ et $F_\psi$ :
\[p\leq \rho_p\leq \frac{q^r}{p^r}.\]
L'encadrement sans l'hypothèse sur $\ell_r$
étant moins précis :
\[p\leq \rho_p\leq q.\]

On peut affiner la méthode pour avoir des
renseignement sur la proportion du nombre de
coefficients non nuls du développement de
Hensel de $\phi_{p,q}(u)$ lorsque $u$ est
un rationnel n'appartenant pas à
$\phi_{p,q}^{-1}({\bf Q})$.
On note $nz(u,k)$ le nombre de coefficients non
nuls dans les $k$ premiers termes du
développement de Hensel de
$\varphi_{p,q}(u)$.
\begin{prop}
	Soient $p<q$ deux nombres premiers fixés.
	Pour tout entier $i$ on note
	$\alpha_i=\max\{k\in{\bf N}, p^k<q^i\}$. 
	Alors pour tout entier $m$,
	\[
		H(u)>m \Rightarrow
		\alpha_{nz(u,m)}+2-m\leq H(u_m)-H(u)\leq
	\alpha_{nz(u,m)}-m-1
\]
\end{prop}

 En effet on a pour $u_m\leq
 \frac{1}{p^m}[q^{nz(u,m)}u+\frac{q^{nz(u,m)}-p^{nz(u,m)}}{q-p}(p-1)p^{m+1-nz(u,m)}]$.
 D'où pour 
 \[H(u)\geq m\Rightarrow 
 p^{H(f^m(u))-1+m}<q^{nz(u,m)}p^{H(u)+1}\]
 et
 \[H(f^m(u))-H(u)\leq \alpha_{nz(u,m}+2-m.\]
La minoration s'obtient de façon analogue.

Donnons enfin un résultat sur le comportement
en moyenne de la quantité $H(f^m(u))-H(u)$.

Il y a modulo $p^{m}$, $(p-1)^i\binom{m}{i}$ classes dont le
développement de Hensel à exactement $i$
coefficients non nuls parmi les $m$ premiers
coefficients. La moyenne de $H(f^m(u))-H(u)$
sur les classes modulos $p^m$ peut donc
s'encadrer  
en utilisant le fait que
$\phi_{p,q}$ induit une bijection sur les
classes modulo $p^m$, et l'encadrement
précédent. 
\begin{theo}
	Soit $m$ un entier positif. Soit $E_m$ un
	ensemble de représentant des classes modulo
	$2^m$, dont les éléments vérifient
	$H(u)>m$. On note $M(E_m)$ la moyenne des
	quantité $H(u_m)-H(u)$ lorsque $u$ parcourt
	$E_m$. Alors
	\[m(\frac{p-1}{p}\frac{\log q}{\log
	p}-1)-2\leq M(E_m)\leq
m(\frac{p-1}{p}\frac{\log q}{\log p}-1)+2\]
\end{theo}
Il s'agit d'estimer
$\frac{1}{p^m}\sum_{i=0}^m\binom{m}{i}(p-1)^i\alpha_i$.
On sait que \[i\frac{\log q}{\log p}-1\leq \alpha_i\leq i\frac{\log q}{\log
p},\]
comme $m(X+1)^{m-1}X=\sum_{i=0}^m
\binom{m}{i} i X^i$, on obtient l'encadrement
voulu.
\begin{coro}
	Si $q^{p-1}<p^p$ alors $M(E_m)$ est
	négatif. Le couple $(p,q)$ est un bon
	candidat pour vérifier $\phi_{p,q}({\bf
	Q})={\bf Q}$.
\end{coro}
Des essais numériques semblent aller dans
cette direction, et écarter les couples
$(p,q)$ ne vérifiant pas cette inégalité.

\section{Généralisation à d'autres anneaux}
\subsection{Cas des anneaux d'entiers}
Dans cette partie on se place dans le cas
d'un anneau d'entiers algébriques monogène,
c'est à dire de la forme $\mathcal{O}_{\bf
K}={\bf Z}[\theta]$.
On sait \cite{MR619892} qu'un tel anneau admet un système de
numération canonique, c'est à dire qu'il
existe $\alpha$ dans l'anneau, tel que 
\[\mathcal{O}_K=\{\sum_{i=0}^n a_i\alpha^i,
	n\in{\bf N}, 0\leq a_i\le
	|\mathrm{N}_{{\bf
	K}/{\bf Q}}\}(\alpha)|\]

Si $\alpha$ est tel que $p=\mathrm{N}_{{\bf
K}/{\bf Q}}(\alpha)$ est un nombre premier
et donne un système de numération canonique,
on peut pour tout nombre entier strictement
positif $q$ minorant
l'ensemble des valeurs absolues des conjugués
de $\alpha$, considérer la fonction 
\[g_{\alpha,q}(u)=\left\{\aligned
	&\frac{u}{\alpha},\text{ si }\alpha|u\\
	&\frac{qu+\epsilon(-qu)}{\alpha},\text{
	sinon}\\
	\endaligned\right.\]
	Où $\epsilon(u)$ désigne l'unique entier
	cmpris entre $0$ et $p-1$, tel que
	$u-\epsilon(u)$ soit divisible par
	$\alpha$.

	Par construction de $g_{\alpha,q}$, et du
	fait du choix de $q$, qui permet de
	controler la taille de l'image par
	$g_{\alpha,q}$ d'un
	élément, la suite des itérés
	$g_{\alpha,q}^m(u)$ est ultimement
	périodique. On obtient donc le  résultat
	suivant :
	\begin{theo}
		Soient $\psi$ une fonction de $\bf N$ dans
		$\overline{\bf N}$, et $a_i$ une suite
		d'éléments de $\{1,\dots p-1\}$.
		La série entière $G(T)=\sum_{n=0}^\infty
		a_n\alpha^{\psi(n)}T^n$ est une fraction
		rationnelle si et seulement si l'élément 
		$G(\frac{1}{q})$ du complété ${\bf
		K}_\alpha$ de $\bf K$
		en $(\alpha)$ est dans $\bf K$.
	\end{theo}

\subsection{Cas de l'anneau des séries formelles
${\bf F}_q[[T]]$}
On considère l'application $\mathcal{S}$ définie sur
${\bf F}_q[[T]]$ par :
$$f\mapsto\begin{cases}
\frac{f}{T},& \text{ si } f(0)=0\cr
\frac{(1+T)f-f(0)}{T},&\text{sinon}\cr
\end{cases}
$$
Et la transformation $\phi$ sur ${\bf
F}_q[[T]]$, définie par 
$$\phi(f)=\sum_{n=0}^\infty \mathcal{S}^n(f)(0)T^n.$$
Notons pour une série 
$\sum_{n=0}^\infty a_n T^n$, 
$r_n(f)$
le nombre de coefficients non nuls d'indice
inférieur ou égal à $n$.
\begin{prop}
La transformation $\phi$ est une isométrie de
${\bf F}_q[[T]]$, d'isométrie réciproque :
$$\phi^{-1}(\sum_{n=0}^\infty
a_nT^n)=\sum_{n=0}^\infty
a_n\frac{T^n}{(1+T)^{r_n(f)}}.$$
\end{prop}
\begin{proof}
La démonstration est analogue à celle faite
pour ${\bf Z}_p$.
\end{proof}
Si on associe à une fraction rationnelle
sous forme irréductible
$P/Q$ de
${\bf F}_p[[T]]$ sa hauteur
$H(P/Q)=\max\{\deg P,\deg Q\}$, l'opérateur
$S$ vérifie naturellement $H(S(P/Q))\leq
H(P/Q)$. Or l'ensemble des fractions
rationnelles de hauteur bornée est fini, donc
la suite $S^n(f)$ est périodique si et
seulement si $f$ est une fraction
rationnelle.

\rotatebox{90}{
\begin{tabular}{|l|l|l|l|}
 \hline
p&q&u&\hfil{\bf Cycle de p\'eriodes
enti\`eres}\hfil\\
\hline
2&3&-17&-25,-37,-55,-82,-41,-61,-91,-136,-68,-34,-17 \\ \hline
3&11&-25&-91,-333,-111,-37,-135,-45,-15,-5,-18,-6,-2,-7,-25\\ \hline
3&13&-47&-203,-879,-293,-1269,-423,-141,-47\\ \hline
5&7&-32&-44,-61,-85,-17,-23,-32\\ \hline
5&13&-2&-5,-1,-2\\ \hline
7&17&-9&-21,-3,-7,-1,-2,-4,-9\\ \hline
7&19&-13&-35,-5,-13\\ \hline
11&13&-17&-20,-23,-27,-31,-36,-42,-49,-57,-67,-79,-93,-109,-128,-151,-178,
-210,-248,-293,-346,-408,-482,-569,\\
 & & &-672,-794,-938,-1108,-1309,-119,-140,-165
,-15,-17\\ \hline
11&19&-13&-22,-2,-3,-5,-8,-13\\ \hline
11&37&-13&-43,-144,-484,-44,-4,-13\\ \hline
13&19&-10&-14,-20,-29,-42,-61,-89,-130,-10\\ \hline
13&47&-10&-36,-130,-10\\ \hline
17&29&-13&-22,-37,-63,-107,-182,-310,-528,-900,-1535,-2618,-154,-262,-446,
-760,-1296,-2210,-130,-221,-13\\ \hline
17&37&-8&-17,-1,-2,-4,-8\\  \hline
17&41&-21&-50,-120,-289,-17,-1,-2,-4,-9,-21\\  \hline
17&73&-4&-17,-1,-4\\ \hline
19&29&-41&-62,-94,-143,-218,-332,-506,-772,-1178\\  \hline
19&83&-74&-323,-17,-74\\ \hline
23&29&-15&-18,-22,-27,-34,-42,-52,-65,-81,-102,-128,-161,-7,-8,-10,-12, -15 \\ \hline
23&53&-20&-46,-2,-4,-9,-20\\ \hline
29&47&-13&-21,-34,-55,-89,-144,-233,-377,-13\\ \hline
37&47&-19&-24,-30,-38,-48,-60,-76,-96,-121,-153,-194,-246,-312,-396,-503,
-638,-810,-1028,-1305,-1657,-2104,\\
 & & &-2672,-3394,-4311,-5476,-148,-4,-5,-6,
-7,-8,-10,-12,-15,-19\\ \hline
41&53&-23&-29,-37,-47,-60,-77,-99,-127,-164,-4,-5,-6,-7,-9,-11,-14,
-18,-23\\ \hline
47&83&-17&-30,-52,-91,-160,-282,-6,-10,-17,-30,-52,-91,-160,-282,-6,-10, 
-17\\ \hline
71&97&-13&-17,-23,-31,-42,-57,-77,-105,-143,-195,-266,-363,-495,-676,-923,
-13\\ \hline
73&97&-23&-30,-39,-51,-67,-89,-118,-156,-207,-275,-365,-5,-6,-7,-9,-11,
-14,-18,-23 \\
\hline
\end{tabular}}

\bibliographystyle{alpha}
\bibliography{./tdn.bib}
\end{document}